\documentclass[11pt]{amsart}
\usepackage[margin=1in]{geometry}

\usepackage{amssymb}
\usepackage{amsthm}
\usepackage{amsmath}
\usepackage{mathrsfs}
\usepackage{amsbsy}
\usepackage{bm}
\usepackage{hyperref}
\usepackage{tikz}
\usepackage{array}
\usepackage{enumerate}
\usepackage{xcolor}
\usepackage{bbm}
\usepackage{comment}
\usepackage{mathtools}

\definecolor{Green}{rgb}{0,0.922,0}
\definecolor{DarkGreen}{rgb}{0,0.5,0}
\definecolor{MildGreen}{rgb}{0,0.784,0}
\definecolor{NormalGreen}{rgb}{0,0.8,0}
\definecolor{Pink}{rgb}{1,0,1}
\definecolor{Cyan}{rgb}{0,1,1}
\definecolor{Yellow}{rgb}{1,1,0}
\definecolor{Gold}{rgb}{1,0.73,0}
\definecolor{lavender}{rgb}{0.45,0,1}

\usepackage[noabbrev,capitalise,nameinlink]{cleveref}

\DeclareFontFamily{U}{mathx}{}
\DeclareFontShape{U}{mathx}{m}{n}{<-> mathx10}{}
\DeclareSymbolFont{mathx}{U}{mathx}{m}{n}
\DeclareMathAccent{\widecheck}{0}{mathx}{"71}

\hypersetup{colorlinks=true, citecolor=lavender, linkcolor=lavender,urlcolor=lavender}

\usepackage{hhline}
\allowdisplaybreaks
\usepackage[noadjust]{cite}

\usepackage{caption}
\usepackage[noabbrev,capitalise,nameinlink]{cleveref}
\crefname{conjecture}{Conjecture}{Conjectures}

\newtheorem{theorem}{Theorem}[section]
\newtheorem{proposition}[theorem]{Proposition}
\newtheorem{corollary}[theorem]{Corollary}

\theoremstyle{definition}

\newtheorem{remark}[theorem]{Remark}
\newtheorem{example}[theorem]{Example}

\newcommand{\dfn}[1]{\textcolor{blue}{\emph{#1}}}

\newcommand{\BPT}{\mathsf{BPT}}
\newcommand{\DBPT}{\mathsf{DBPT}}
\newcommand{\FBPT}{\mathsf{FBPT}}
\newcommand{\Branch}{\mathsf{Branch}}
\newcommand{\A}{\mathcal{A}}
\newcommand{\T}{\mathbf{T}}
\newcommand{\ttt}{\boldsymbol{\tau}}
\newcommand{\N}{\mathbb{N}}
\newcommand{\1}{\mathbbm{1}}
\newcommand{\Sn}{\mathfrak{S}_n}
\newcommand{\Dn}{\mathfrak{D}_n}
\newcommand{\rr}{\mathrm{right}}
\newcommand{\pp}{\mathrm{two}}

\let\P\relax
\newcommand{\I}{\mathcal{I}}
\renewcommand{\L}{\mathcal{L}}
\newcommand{\P}{\mathcal{P}}
\newcommand{\IF}{\mathsf{IF}}

\newcommand{\NC}{\mathrm{NC}}
\newcommand{\II}{\mathrm{Int}}
\newcommand{\Des}{\mathrm{Des}}
\newcommand{\des}{\mathrm{des}}
\newcommand{\runs}{\mathrm{druns}}
\newcommand{\sw}{\mathrm{SE}} 
\newcommand{\Tan}{a}

\begin{document}

\title[]{Boolean, Free, and Classical Cumulants \\ as Tree Enumerations}
\subjclass[2010]{}

\author[]{Colin Defant}
\address[]{Department of Mathematics, Harvard University, Cambridge, MA 02138, USA}
\email{colindefant@gmail.com}

\author[]{Mitchell Lee}
\address[]{Department of Mathematics, Harvard University, Cambridge, MA 02138, USA}
\email{mitchell@math.harvard.edu}

\begin{abstract}
Defant found that the relationship between a sequence of (univariate) classical cumulants and the corresponding sequence of (univariate) free cumulants can be described combinatorially in terms of families of binary plane trees called \emph{troupes}. Using a generalization of troupes that we call \emph{weighted troupes}, we generalize this result to allow for multivariate cumulants. Our result also gives a combinatorial description of the corresponding Boolean cumulants. This allows us to answer a question of Defant regarding his \emph{troupe transform}. We also provide explicit distributions whose cumulants correspond to some specific weighted troupes. 
\end{abstract} 

\maketitle

\section{Introduction}\label{sec:intro}
Let $\mathcal A$ be a unital associative algebra over a commutative ring $\mathbb K$, and let $\varphi\colon\mathcal A\to\mathbb K$ be a unital linear functional. The pair $(\mathcal{A},\varphi)$ is called a  \dfn{noncommutative probability space} (over $\mathbb K$), and the functional $\varphi$ is called a \dfn{noncommutative expectation}. Let $(X_1, \ldots, X_n)$ be a tuple of elements of $\mathcal A$. We associate to this tuple three important quantities: the \emph{classical cumulant} $K_n(X_1, \ldots, X_n)$, the \emph{free cumulant} $R_n(X_1, \ldots, X_n)$, and the \emph{Boolean cumulant} $B_n(X_1, \ldots, X_n)$. The classical cumulant $K_n(X_1, \ldots, X_n)$ is informally a measure of the dependence of $X_1, \ldots, X_n$, whereas the free cumulant $R_n(X_1, \ldots, X_n)$ is an analogue used in the free probability theory of Voiculescu \cite{MingoSpeicher2017,Voiculescu1,Voiculescu2,Voiculescu3,Voiculescu4}. Classical, free, and Boolean cumulants are all related by formulas involving sums over particular sets of partitions of the set $[n]:=\{1, \ldots, n\}$ \cite{Arizmendi2015}.

If the variables $X_1, \ldots, X_n$ are all equal, then more is known. In 2013, Josuat-Verg\`es found a combinatorial expansion of the negative classical cumulant $-K_n(X, \ldots, X)$ into negative free cumulants $-R_i(X, \ldots, X)$ for $i \leq n$ \cite{Josuat2013}. In 2022, Defant proved that the coefficients of Josuat-Verg\`es's expansion count objects called \emph{valid hook configurations}, which also appear in the study of West's stack-sorting map \cite{Defant2022}. This connection allowed Defant to prove several results about the stack-sorting map. 

Defant also found that the relationship between free and classical cumulants is encoded by special families of binary plane trees called \emph{troupes}. Our main theorem (\cref{thm:big_theorem}) generalizes this result in three ways. First, our theorem deals with a new generalization of troupes that we call \emph{weighted troupes}. Second, our theorem does not require the elements $X_1,\ldots,X_n\in\A$ to be all equal. Third, our theorem handles Boolean cumulants in addition to classical and free cumulants. 

A troupe is a family of binary plane trees satisfying certain conditions related to an operation called \emph{insertion}. A \emph{branch} is a binary plane tree in which no vertex has more than one child. In \cite{Defant2022}, Defant proved that each troupe is uniquely determined by the set of branches that it contains and that every set of branches \emph{generates} a unique troupe. He then defined the \emph{troupe transform} to be the transform that takes as input a sequence $(\omega_n)_{n\geq 1}$ enumerating a set of branches and outputs the sequence $(\widecheck\omega_n)_{n\geq 1}$ enumerating the troupe generated by that set of branches. He asked what could be said about the relationship between the generating functions \[\mathscr B(x)=\sum_{n\geq 1}\omega_nx^n\quad\text{and}\quad \mathscr T(x)=\sum_{n\geq 1}\widecheck\omega_nx^n\] (\cite[Question~9.1]{Defant2022}). In particular, he asked what conditions on $\mathscr B(x)$ would guarantee that $\mathscr T(x)$ is algebraic. Using \cref{thm:big_theorem}, we completely resolve these questions. Namely, we show in \cref{thm:troupe-transform} that 
\[\mathscr T(x) = \mathscr B\left(\frac{x}{1 - x\mathscr T(x)}\right);\] this implies (see \cref{cor:algebraic}) that $\mathscr{T}(x)$ is algebraic if and only if $\mathscr{B}(x)$ is algebraic. (In fact, our \cref{thm:troupe-transform} is more general because it deals with weighted troupes instead of ordinary troupes.) 

\cref{sec:background} provides preliminaries concerning partitions, cumulants, and binary plane trees. In \cref{sec:insertion}, we define and provide examples of weighted troupes, and we discuss some of their properties. In particular, we show how a binary plane tree can be decomposed into branches called its \emph{insertion factors}. In \cref{sec:main_theorem}, we state and prove our main theorem (\cref{thm:big_theorem}) regarding weighted troupes and cumulants. In \cref{sec:troupe_transform}, we prove \cref{thm:troupe-transform}, which completely describes the troupe transform. \cref{sec:peaks} is devoted to explaining how one can easily compute the insertion factors of a decreasing binary plane tree directly from the permutation corresponding to it via the inorder bijection. Finally, in \cref{sec:distributions}, we compute explicit distributions whose cumulants correspond to some especially notable weighted troupes.

\section{Background}\label{sec:background}
\subsection{Partitions}
Let $Y$ be a finite set. A \dfn{partition} of $Y$ is a collection of nonempty disjoint sets, called \dfn{blocks}, whose union is $Y$. For any partition $\pi$ of $Y$, we define an equivalence relation $\sim_\pi$ on $Y$, where $i \sim_\pi j$ if and only if $i$ and $j$ are in the same block of $\pi$. In this way, partitions of $Y$ are in bijective correspondence with equivalence relations on $Y$. 

Suppose that $Y$ is totally ordered. We say that a partition $\pi$ of $Y$ is an \dfn{interval partition} if there do not exist $i, j, k \in Y$ with $i < j < k$ such that $i \sim_\pi k\not\sim_\pi j$. We say that $\pi$ is \dfn{noncrossing} if there do not exist $i, j, k, \ell \in Y$ with $i < j < k < \ell$ such that $i \sim_\pi k$, $j \sim_\pi \ell$, and $i \not \sim_\pi j$. We say that a noncrossing partition $\pi$ of $Y$ is \dfn{irreducible} if $\min Y \sim_\pi \max Y$. We write $\Pi(n)$, $\II(n)$, $\NC(n)$, and $\NC_{\text{irr}}(n)$ for the set of partitions of $[n]$, the set of interval partitions of $[n]$, the set of noncrossing partitions of $[n]$, and the set of irreducible noncrossing partitions of $[n]$, respectively. 

Let $\Sn$ denote the symmetric group whose elements are permutations of $[n]$. The \dfn{descent set} of a permutation ${\sigma = \sigma(1) \cdots \sigma(n) \in \Sn}$ is $\Des(w) = \{i \in [n-1] : \sigma(i) > \sigma(i+1)\}$. Let ${\des(\sigma)= |\Des(\sigma)|}$. The \dfn{descending runs} of $\sigma$ are the maximal consecutive decreasing subsequences of $\sigma$; note that $\sigma$ has $n-\des(\sigma)$ descending runs. Let $\runs(\sigma)$ be the partition \[\{D_1(\sigma),\ldots,D_{n-\des(\sigma)}(\sigma)\}\in\Pi(n),\] where $D_j(\sigma)$ is the set of numbers appearing in the $j$th descending run of $\sigma$. For example, 
\[\runs(854791632)=\{\{4,5,8\},\{7\},\{1,9\},\{2,3,6\}\}\in\Pi(9).\]

\subsection{Cumulants}
Let $(\A,\varphi)$ be a noncommutative probability space over a commutative ring $\mathbb K$. For $n \geq 1$, there are three multilinear functionals $K_n, R_n, B_n \colon \A^n \to \mathbb K$, called the \dfn{classical cumulant}, \dfn{free cumulant}, and \dfn{Boolean cumulant}, respectively.

For any set $U = \{u_1, \ldots, u_k\} \subseteq [n]$ with $u_1 < \cdots < u_k$ and any $X_1, \ldots, X_n \in \A$, define $K_U(X_1, \ldots, X_n) = K_k(X_{u_1}, \ldots, X_{u_k})$, and similarly define $R_U$ and $B_U$. For $\pi\in\Pi(n)$, define \[K_\pi(X_1, \ldots, X_n) = \prod_{U \in \pi} K_U(X_1, \ldots, X_n),\] and similarly define $R_\pi$ and $B_\pi$.

The functions $K_n,R_n,B_n\colon\A^n\to\mathbb K$ are uniquely determined by the following equations \cite[Theorem~2.5]{Arizmendi2015}:
\begin{align}
    \label{eq:classical}\varphi(X_1 \cdots X_n) &= \sum_{\pi \in \Pi(n)} K_\pi(X_1, \ldots, X_n), \\
    \label{eq:free} \varphi(X_1 \cdots X_n) &= \sum_{\pi \in \NC(n)} R_\pi(X_1, \ldots, X_n), \\
    \label{eq:Boolean} \varphi(X_1 \cdots X_n) &= \sum_{\pi \in \II_n} B_\pi(X_1, \ldots, X_n).
\end{align}

Let $X \in \A$. For $n \geq 0$, the $n$th \dfn{moment} of $X$ is $\varphi(X^n)$. The \dfn{moment-generating function} of $X$ is \[M_X(t) = \sum_{n \geq 0} \varphi(X^n) \frac{t^n}{n!},\] the exponential generating function of the moments of $X$. This is well defined as long as $\mathbb{K}$ contains $\mathbb{Q}$. The classical cumulants $K_n(X, \ldots, X)$ can be computed from $M_X(t)$ via the equation
\begin{equation}\label{eq:classical-egf}
    \sum_{n \geq 1} K_n(X, \ldots, X) \frac{t^n}{n!} = \log M_X(t).
\end{equation}

\subsection{Binary Plane Trees}\label{subsec:bpt}

A \dfn{binary plane tree} is a rooted tree $T$ in which each child of a vertex is designated as a left child or a right child and no vertex can have more than one left child or more than one right child. The empty binary plane tree is denoted $\varnothing$. When there is no risk of confusion, we also use the symbol $T$ to denote the set of vertices of the tree $T$. The \dfn{size} of $T$, denoted $|T|$, is the number of vertices in $T$. Let $\BPT$ denote the set of all binary plane trees (up to isomorphism), and let $\BPT_n$ denote the set of all binary plane trees of size $n$. It is a classical result that $|\BPT_n| = C_n$, where $C_n = \frac{1}{n+1}\binom{2n}{n}$ denotes the $n$th Catalan number. 

The \dfn{inorder} on a binary plane tree $T$ is the unique total order on the vertices of $T$ such that for each vertex $v$ of $T$, all the vertices in the left subtree of $v$ precede $v$, which in turn precedes all the vertices in the right subtree of $v$. The \dfn{postorder} on a binary plane tree $T$ is the unique total order on the vertices of $T$ such that for each vertex $v$ of $T$, all the vertices in the left subtree of $v$ precede all the vertices in the right subtree of $v$, which in turn precede $v$.

A \dfn{labeling} of a binary plane tree $T$ is a function $\L \colon T \to \N$. The \dfn{label} of a vertex $v$ under a labeling $\L$ is the integer $\L(v)$. We say that a labeling of $T$ is \dfn{standard} if it uses each of the labels $1, \ldots, |T|$ exactly once. A labeling of a binary plane tree is \dfn{decreasing} if the label of each vertex is greater than the labels of its children.

Given a binary plane tree $T$, we will now describe two important standard labelings of $T$. The \dfn{inorder labeling} of $T$, denoted $\I_T$, assigns the vertices of $T$ the labels $1, \ldots, n$ in inorder. The \dfn{postorder labeling} of $T$, denoted $\P_T$, assigns the vertices of $T$ the labels $1, \ldots, n$ in postorder. Note that the postorder labeling of $T$ is decreasing. See \cref{fig:inorder_postorder}. 

Let $\DBPT$ denote the set of all trees $T \in \BPT$ that are equipped with a standard decreasing labeling $\L$. Formally, $\DBPT$ is the set of all pairs $(T, \L)$, where $T \in \BPT$ and $\L$ is a standard decreasing labeling of $T$. However, we will usually refer to an element of $\DBPT$ using just the letter $T$ and leave the labeling $\L$ implicit. We call the elements of $\DBPT$ \dfn{decreasing binary plane trees}. Let $\DBPT_n$ denote the set of all decreasing binary plane trees of size $n$. There is an injection $\iota \colon \BPT \to \DBPT$ that equips each binary plane tree $T$ with its postorder labeling.

A \dfn{branch} is a nonempty binary plane tree in which every vertex has at most one child. Let $\Branch$ denote the set of all branches, and let $\Branch_n$ denote the set of all branches of size $n$. Clearly, we have $|\Branch_{n}| = 2^{n-1}$ for $n \geq 1$. 

\begin{figure}
  \begin{center}{\includegraphics[height=3.812cm]{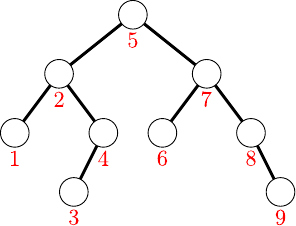}}\qquad\qquad\qquad\includegraphics[height=3.812cm]{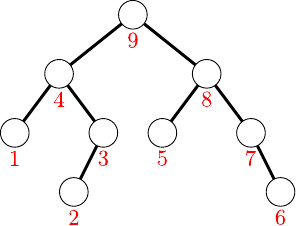}
  \end{center}
\caption{The inorder labeling (left) and postorder labeling (right) of a binary plane tree.}\label{fig:inorder_postorder}
\end{figure}

We now define functions $\alpha, \beta \colon \DBPT_n \to \Sn$. Given $(T,\L) \in \DBPT_n$ and a label $k \in [n]$, let \[\alpha(T,\L)(k) = \L(\I_T^{-1}(k))\quad\text{and}\quad\beta(T,\L)(k) = \L(\P_T^{-1}(k)).\] In other words, $\alpha(T)$ is the permutation formed by reading the labels of the vertices of $T$ in inorder, while $\beta(T)$ is the permutation formed by reading the labels of the vertices of $T$ in postorder.

It is well known that $\alpha$ is a bijection. West's \dfn{stack-sorting map} is the function $s \colon \Sn \to \Sn$ given by \[s(\sigma) = \beta(\alpha^{-1}(\sigma)).\]

Fix an index set $I$ (possibly infinite). An \dfn{$I$-coloring} of a binary plane tree $T$ is a function $\chi\colon T\sqcup\{\boxminus\}\to I$, where $\boxminus$ is a special symbol that denotes an element not in $T$. Let $\BPT(I)$ denote the set of all $I$-colored binary plane trees, let $\Branch(I)$ denote the set of all $I$-colored branches, and let $\DBPT(I)$ denote the set of all $I$-colored binary plane trees equipped with a decreasing labeling. We have an injection $\iota \colon \BPT(I) \to \DBPT(I)$ that assigns each $I$-colored binary plane tree the postorder labeling.

Much of the flavor of our results is retained if $I = \{\star\}$ is a singleton, so the reader may find it helpful to keep that special case in mind. In that case, $\BPT(I)$, $\Branch(I)$, and $\DBPT(I)$ are in obvious bijection with $\BPT$, $\Branch$, and $\DBPT$, respectively. When there is no risk of confusion, we will sometimes use the notations $\BPT$, $\Branch$, and $\DBPT$ as shorthand for $\BPT(\{\star\})$, $\Branch(\{\star\})$, and $\DBPT(\{\star\})$, respectively.

Given $i_1, \ldots, i_n \in I$, define $\DBPT(i_1, \ldots, i_n)$ to be the set of all $T \in \DBPT(I)$ such that $\chi(\boxminus) = i_n$ and such that $\chi(v) = i_{\L(v)}$ for all $v \in T$.

Define $\BPT(i_1, \ldots, i_n) = \iota^{-1}(\DBPT(i_1, \ldots, i_n))$. This is the set of all $I$-colored binary plane trees such that the colors of the vertices form the sequence $i_1, \ldots, i_{n-1}$ when read in postorder and such that the color of $\boxminus$ is $i_n$.

Define $\Branch(i_1, \ldots, i_n) = \Branch(I) \cap \BPT(i_1, \ldots, i_n)$.

Given a tree $T \in \DBPT$ and a vertex $v \in \DBPT$ with exactly one child, we can \dfn{swing} $T$ at $v$ by changing the subtree of $v$ from a left subtree to a right subtree or vice versa. In a similar way, we may swing trees in $\BPT$, $\Branch$, $\BPT(I)$, $\Branch(I)$, $\DBPT(I)$, $\BPT(i_1, \ldots, i_n)$, $\Branch(i_1, \ldots, i_n)$, and $\DBPT(i_1, \ldots, i_n)$.

\section{Insertion and Troupes}\label{sec:insertion}

Let $T_1, T_2 \in \BPT(I) \setminus \{\varnothing\}$, and let $v$ be a vertex of $T_1$. We will now define a new $I$-colored tree $\nabla_v(T_1, T_2) \in \BPT(I) \setminus \{\varnothing\}$, called the \dfn{insertion} of $T_2$ into $T_1$ at $v$. First, form a new tree $T_1^v$ from $T_1$ by extending $v$ into a left edge; identify $v$ with the bottom vertex of this left edge, and call the upper vertex $v^*$. Next, attach $T_2$ as the right subtree of $v^*$. To define the $I$-coloring $\chi$ of $\nabla_v(T_1, T_2)$, let $\chi_1$ and $\chi_2$ be the $I$-colorings of $T_1$ and $T_2$, respectively. For $u\in T_1\sqcup\{\boxminus\}$, let $\chi(u)=\chi_1$. For $u'\in T_2$, let $\chi(u')=\chi_2(u')$. Finally, let $\chi(v^*)=\chi_2(\boxminus)$. 

\begin{example}
Let $I$ be a $3$-element set whose elements are represented by the colors {\color{Cyan}cyan}, {\color{Yellow}yellow}, and {\color{Pink}magenta}. Consider the $I$-colored binary plane trees $T_1$ and $T_2$ shown on the left of \cref{fig:insertion}, where the color of $\boxminus$ in each $I$-colored tree is represented by a colored version of the symbol $\boxminus$ floating above and to the left of the root of the tree. Let $v\in T_1$ be as indicated. Then the insertion of $T_2$ into $T_1$ at $v$ is the $I$-colored binary plane tree shown on the right of \cref{fig:insertion}. 
\end{example}

\begin{figure}
  \begin{center}{\includegraphics[height=9.087cm]{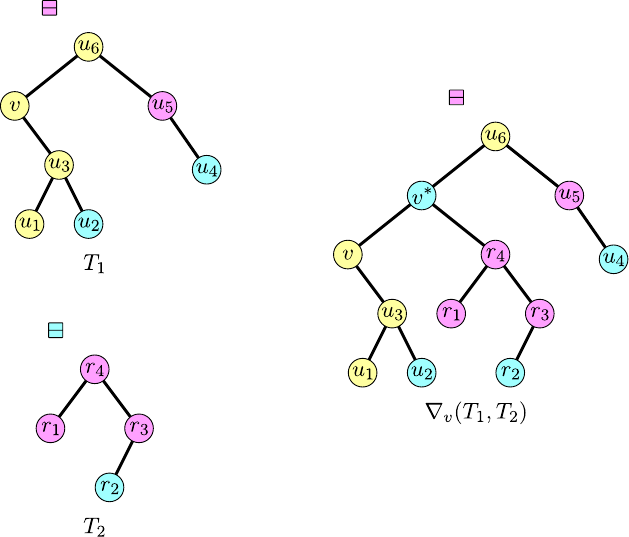}}
  \end{center}
\caption{Two $I$-colored binary planes trees $T_1$ and $T_2$ (left) and the insertion of $T_2$ into $T_1$ at the vertex $v$ (right).}\label{fig:insertion} 
\end{figure} 

Every tree $T \in \BPT(I) \setminus \{\varnothing\}$ can be formed from branches by repeatedly performing the insertion operation. We now describe how this is done. 
Let $\Lambda_T$ be the set consisting of $\boxminus$ and the vertices of $T$ with two children. Define the function $\gamma \colon T \cup \{\boxminus\} \to \Lambda_T$ as follows. If $u \in \Lambda_T$, then let $\gamma(u) = u$. If $u \not \in \Lambda_T$ and there is a vertex $v \in \Lambda_T \setminus \{\boxminus\}$ such that the right subtree of $v$ contains $u$, then let $\gamma(u)$ be the lowest such vertex $v$. Otherwise, let $\gamma(u) = \boxminus$. Define $\pi_T = \{\gamma^{-1}(v) : v \in \Lambda_T\}$; this is a partition of $T \cup \{\boxminus\}$.

Let $v \in \Lambda_T$. Define a binary plane tree $T_{v} \in \BPT(I) \setminus \{\varnothing\}$ on the vertex set $\gamma^{-1}(v) \setminus \{v\}$ as follows. For any $u, w \in \gamma^{-1}(v) \setminus \{v\}$, we have that $w$ is a left (respectively, right) child of $u$ in $T_{v}$ if and only if $w$ appears in the left (respectively, right) subtree of $u$ in $T$ and there is no element of $\gamma^{-1}(v)$ between $w$ and $u$ in $T$. Clearly, $T_{v}$ is a branch. The $I$-coloring of $T_{v}$ is defined as follows. For $u \in \gamma^{-1}(v) \setminus \{v\}$, the color of $u$ in $T_{v}$ is the same as its color in $T$. The color of $\boxminus$ in $T_{v}$ is the color of $v$ in $T$. We call the trees $T_{v} \in \Branch(I)$ for $v \in \Lambda_T$ the \dfn{insertion factors} of $T$. 

\begin{example}\label{exam:insertion_factors}
Let $T$ be the $I$-colored binary plane tree shown on the top of \cref{fig:if}. We have $\Lambda_T=\{r_4,r_8,r_9,\boxminus\}$, so $\pi_T=\{\gamma^{-1}(r_4),\gamma^{-1}(r_8),\gamma^{-1}(r_9),\gamma^{-1}(\boxminus)\}$, where 
\[\gamma^{-1}(r_4)=\{r_2,r_3,r_4\},\quad \gamma^{-1}(r_8)=\{r_6,r_7,r_8\},\quad \gamma^{-1}(r_9)=\{r_5,r_9\},\quad \gamma^{-1}(\boxminus)=\{r_1,\boxminus\}.\] The insertion factors of $T$ are $T_{r_4}$, $T_{r_8}$, $T_{r_9}$, and $T_\boxminus$, which are shown on the bottom of \cref{fig:if}. 
\end{example} 

\begin{figure}
  \begin{center}{\includegraphics[height=7.682cm]{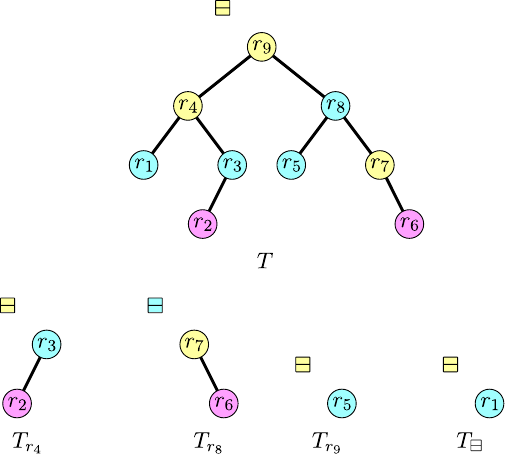}}
  \end{center}
\caption{An $I$-colored binary plane tree (top) and its insertion factors (bottom).}\label{fig:if} 
\end{figure}

Note that multiple insertion factors of $T$ may be isomorphic to each other. (In \cref{exam:insertion_factors}, $T_{r_9}$ and $T_\boxminus$ are isomorphic.) The collection of insertion factors of $T$, up to isomorphism, is a multiset $\IF(T) \subseteq \BPT(I)$. The following proposition is immediate from the definitions.

\begin{proposition}\label{prop:if}
    Let $T_1, T_2 \in \BPT(I) \setminus \{\varnothing\}$, and let $v$ be a vertex of $T_1$. Then $\IF(\nabla_v(T_1, T_2))$ is the multiset union of $\IF(T_1)$ and $\IF(T_2)$.
\end{proposition}

\Cref{prop:if} implies that every tree can be formed from its insertion factors by repeatedly performing the insertion operation.

\begin{remark}
Let us make an analogy between binary plane trees and classical algebraic structures (such as modules over a ring). One should view $I$-colored branches as the ``indecomposable'' objects of $\BPT(I)\setminus \{\varnothing\}$. Each tree $T\in\BPT(I) \setminus \{\varnothing\}$ can be obtained from a sequence of $I$-colored branches via iterated insertions. Such a sequence of $I$-colored branches is not necessarily unique. However, \cref{prop:if} implies that the multiset of $I$-colored branches appearing in the sequence is uniquely determined by $T$ (it is $\IF(T)$). Thus, one can view \cref{prop:if} as an analogue of the Jordan--H\"older theorem for binary plane trees.   
\end{remark}

A \dfn{troupe} is a subset $\T \subseteq \BPT(I)\setminus\{\varnothing\}$ such that for all $T_1, T_2 \in \BPT(I)$ and all vertices $v \in T_1$, we have $\nabla_v(T_1, T_2) \in \T$ if and only if $T_1 \in \T$ and $T_2 \in \T$. Troupes were introduced in \cite{Defant2022}; we generalize them as follows. Let $\mathbb K$ be a commutative ring. A \dfn{weighted troupe} is a function $\ttt \colon \BPT(I) \to \mathbb K$ with the following properties:
\begin{itemize}
\item $\ttt(\varnothing)=0$;
\item For all $T_1, T_2 \in \BPT(I)$ and all vertices $v \in T_1$, we have $\ttt(\nabla_v(T_1, T_2)) = \ttt(T_1)\ttt(T_2)$.
\end{itemize} 
Observe that the indicator function of a troupe is a weighted troupe.

Let $\ttt\vert_{\Branch(I)}$ denote the restriction of a weighted troupe $\ttt$ to $\Branch(I)$. The following result is a generalization of \cite[Theorem~2.3]{Defant2022} to weighted troupes.

\begin{proposition}\label{thm:JordanHolder}
The map $\ttt\mapsto\ttt\vert_{\Branch(I)}$ is a bijection from the set of weighted troupes to the set of functions from $\Branch(I)$ to $\mathbb K$.
\end{proposition}
\begin{proof}
Let $\ttt \colon \BPT(I) \to \mathbb K$ be a function. By \cref{prop:if}, we have that $\ttt$ is a weighted troupe if and only if $\ttt(\varnothing) = 0$ and \[\ttt(T) = \prod_{T' \in \IF(T)} \ttt(T')\] for all $T \in \BPT(I) \setminus \{\varnothing\}$. The proposition follows.
\end{proof}

For $\ttt \colon \BPT(I) \to \mathbb K$ and $T \in \DBPT(I)$, we will, slightly abusing notation, write $\ttt(T)$ to denote the result of applying $\ttt$ to the underlying (unlabeled) tree of $T$.

Let us briefly mention some examples of troupes and weighted troupes. We will revisit some of these examples in \cref{sec:distributions}.
\begin{enumerate}
\item The simplest example of a troupe is $\BPT(I)\setminus\{\varnothing\}$; its set of branches is the entire set $\Branch(I)$. \item A nonempty binary plane tree is called \dfn{full} if no vertex has exactly one child. The set of full $I$-colored binary plane trees is a troupe; its set of branches is $\Branch_1(I)$. 
\item A \dfn{Motzkin tree} is a nonempty binary plane tree in which every vertex that has a right child also has a left child. The set of $I$-colored Moztkin trees is a troupe; its set of branches is precisely the set of nonempty $I$-colored branches with no right edges. 
\item Fix a subset $J\subseteq I$. Let ${\bf T}$ be the set of nonempty $I$-colored binary plane trees such that every vertex with a left child is assigned a color from $J$ and $\boxminus$ is also assigned a color from $J$. Then ${\bf T}$ is a troupe. 
\item Let $t_1$ and $t_2$ be indeterminates. Let $\rr(T)$ and $\pp(T)$ denote the number of right edges and the number of vertices with two children, respectively, in a binary plane tree $T$. If ${\bf T}$ is a troupe, then the map $\ttt\colon\BPT(I)\to\mathbb C[t_1,t_2]$ defined by 
\[\ttt(T)=\1_{\T}(T)t_1^{\rr(T)+1}t_2^{\pp(T)+1}\] is a weighted troupe. 
\item Let $t$ be an indeterminate, and let $J\subseteq I$. Let $\T$ be a troupe, and define $\ttt\colon\BPT(I)\to\mathbb C[t]$ by 
\[\ttt(T)=\1_{\T}(T)t^{|\chi^{-1}(J)|},\] where $\chi$ is the $I$-coloring of $T$. Then $\ttt$ is a weighted troupe. 
\end{enumerate}

\section{Boolean, Free, and Classical Cumulants as Tree Enumerations}\label{sec:main_theorem}

The following theorem generalizes \cite[Theorem~6.1]{Defant2022} in three ways. First, it deals with weighted troupes rather than (ordinary) troupes. Second, it involves potentially mixed cumulants rather than only cumulants of a single variable. Third, it includes a condition about Boolean cumulants in addition to the conditions about classical and free cumulants. If $I = \{\star\}$ is a singleton, 
then we recover \cite[Theorem~6.1]{Defant2022} from parts \eqref{item1} and \eqref{item2} in the following theorem. We encourage the reader to refer to \cref{exam:Psi,exam:Phi} while reading the following proof. 

\begin{theorem}\label{thm:big_theorem}
    Let $(\A, \varphi)$ be a noncommutative probability space over a commutative ring $\mathbb K$. Let $I$ be an index set, and let $(X_i)_{i \in I}$ be a sequence of elements of $\A$ indexed by the set $I$. Let $\ttt \colon \BPT(I) \to \mathbb K$ be a weighted troupe. The following are equivalent:
    \begin{enumerate}[(i)]
        \item\label{item1} For all $n \geq 1$ and all $i_1, \ldots, i_n \in I$, we have \[-K_n(X_{i_1}, \ldots, X_{i_n}) = \sum_{T \in \DBPT(i_1, \ldots, i_n)} \ttt(T).\]
        \item\label{item2} For all $n \geq 1$ and all $i_1, \ldots, i_n \in I$, we have \[-R_n(X_{i_1}, \ldots, X_{i_n}) = \sum_{T \in \BPT(i_1, \ldots, i_n)} \ttt(T).\] 
        \item\label{item3} For all $n \geq 1$ and all $i_1, \ldots, i_n \in I$, we have \[-B_n(X_{i_1}, \ldots, X_{i_n}) = \sum_{T \in \Branch(i_1, \ldots, i_n)} \ttt(T).\] 
    \end{enumerate}
\end{theorem}

\begin{proof}
    For any subset $U = \{u_1, \ldots, u_k\} \subseteq [n]$ with $u_1 < \cdots < u_k$, let $i_U=(i_{u_1}, \ldots, i_{u_k})$.

    By \eqref{eq:classical}, \eqref{eq:free}, and \eqref{eq:Boolean}, we have $K_1(X) = R_1(X) = B_1(X) = \varphi(X)$ for all $X \in \A$. Therefore (since $\ttt(\varnothing) = 0$), each of \eqref{item1}, \eqref{item2}, and \eqref{item3} implies that ${K_1(X_i) = R_1(X_i) = B_1(X_i) = \varphi(X_i) = 0}$ for all $i \in I$. Hence, we may assume in what follows that $K_1(X_i) = R_1(X_i) = B_1(X_i) = \varphi(X_i) = 0$ for all $i \in I$.
    
    First, we will show that \eqref{item2} is equivalent to \eqref{item3}. Fix $n\geq 1$ and $i_1,\ldots,i_n\in I$. By \cite[Equation~1.4]{Arizmendi2015}, we have that
    \begin{equation}\label{eq:negative-r}
    -R_n(X_{i_1}, \ldots, X_{i_n}) = \sum_{\pi \in \NC_{\text{irr}}(n)} \prod_{U \in \pi}(-B_U(X_{i_1}, \ldots, X_{i_n})).
    \end{equation}
    Therefore, it suffices to show that
    \begin{equation}\label{eq:bpt-branch}
        \sum_{T \in \BPT(i_1, \ldots, i_n)} \ttt(T) = \sum_{\pi \in \NC_{\text{irr}}(n)} \prod_{U \in \pi} \sum_{T \in \Branch(i_U)} \ttt(T).
    \end{equation}
    Let $\NC_{\text{irr}}^{\geq 2}(n)$ be the set of all irreducible noncrossing partitions of $[n]$ whose blocks all have cardinality at least $2$. Observe that the outer summand on the right-hand side of \eqref{eq:bpt-branch} vanishes for $\pi \not \in \NC_{\text{irr}}^{\geq 2}(n)$. Let $P(i_1,\ldots,i_n)$ be the set of pairs $(\pi, (T_U)_{U \in \pi})$ with $\pi \in \NC_{\text{irr}}^{\geq 2}(n)$ and $T_U \in \Branch(i_U)$ for all $U \in \pi$. The right-hand side of \eqref{eq:bpt-branch} can be expanded to
    \[\sum_{(\pi, (T_U)_{U \in \pi}) \in P(i_1, \ldots, i_n)} \prod_{U \in \pi} \ttt(T_U).\]
    Therefore, to prove \eqref{eq:bpt-branch}, it suffices to exhibit a bijection $\Psi\colon P(i_1,\ldots,i_n)\to\BPT(i_1, \ldots, i_n)$ such that if $\Psi(\pi, (T_U)_{U \in \pi})=T$, then the multiset of branches $T_U$ for $U\in\pi$ is $\IF(T)$.
    
    Let $(\pi, (T_U)_{U \in \pi}) \in P(i_1, \ldots, i_n)$. For each $U\in\pi$, we may assume that the vertex set of the branch $T_U$ is $U\setminus\max U$ and that each non-root vertex of $T_U$ is smaller than its parent. The binary plane tree $\Psi(\pi, (T_U)_{U \in \pi}) \in \BPT(i_1, \ldots, i_n)$ can be constructed as follows.
    \begin{itemize}
        \item The vertices of $\Psi(\pi, (T_U)_{U \in \pi})$ are the integers $1, \ldots, n-1$.
        \item Let $j \in [n - 1]$, and let $U \in \pi$ be the block containing $j$.
        \begin{itemize}
            \item If $j = \min U$, then $j$ is a leaf of $\Psi(\pi, (T_U)_{U \in \pi})$.
            \item 
            If $\min U<j<\max U$, then $j-1$ is the unique child of $j$ in $\Psi(\pi, (T_U)_{U \in \pi})$; moreover, $j-1$ is a left (respectively, right) child of $j$ in $\Psi(\pi, (T_U)_{U \in \pi})$ if and only if $j$ has a left (respectively, right) child in $T_U$.  
            \item If $j = \max U$, then $j$ has two children in $\Psi(\pi, (T_U)_{U \in \pi})$; the right child of $j$ is $j-1$, and the left child of $j$ is $\min U - 1$.
        \end{itemize}
        \item The $I$-coloring is given by $\chi(j) = i_j$ for $j \in [n-1]$ and $\chi(\boxminus) = i_n$. 
    \end{itemize}
    We need to check that $\Psi(\pi, (T_U)_{U \in \pi})$ is actually a binary plane tree whose postorder agrees with the usual total order on $[n]$. To do so, we describe a recursive procedure for constructing $\Psi(\pi, (T_U)_{U \in \pi})$. Let $U_1,\ldots,U_r$ be the blocks of $\pi$, listed so that $\min U_1<\cdots<\min U_r$. For $2\leq i\leq r$, let $v_i=\min U_i-1$. Let $\Upsilon_1=T_{U_1}$. For $2\leq i\leq r$, let $\Upsilon_i=\nabla_{v_i}(\Upsilon_{i-1},T_{U_i})$, and rename the new vertex $v_i^*$ created during this insertion as $\max U_i$ (so $\max U_i$ is the parent of $\min U_i-1$ in $\Upsilon_i$). It is straightforward to check that $\Psi(\pi, (T_U)_{U \in \pi})=\Upsilon_r$. The fact that $\Upsilon_r$ is a binary plane tree whose postorder agrees with the usual total order on $[n]$ follows from the definition of insertion and the fact that $\pi$ is an irreducible noncrossing partition.  
    
    Inversely, let $T \in \BPT(i_1, \ldots, i_n)$. We will now describe $\Psi^{-1}(T) \in P(i_1, \ldots, i_n)$. Using the notation from \cref{sec:insertion}, we have that $\pi_T = \{\gamma^{-1}(v) : v \in \Lambda_T\}$ is a partition of $T \cup \{\boxminus\}$. Moreover, the blocks of $\pi_T$ correspond to the insertion factors of $T$. By mapping $T$ to $[n-1]$ using the postorder labeling and mapping $\boxminus$ to $n$, we may transport $\pi_T$ to a partition $\pi \in \Pi(n)$. Again, each part $U \in \pi$ has a corresponding insertion factor $T_U \in \Branch(I)$. Let $\Psi^{-1}(T) = (\pi, (T_U)_{U \in \pi})$. The fact that $\pi$ is noncrossing is immediate from the definition of the postorder and the definition of $\gamma$. Because each insertion factor contains at least one vertex together with the symbol $\boxminus$, every block of $\pi$ has cardinality at least $2$. To see that $\pi$ is irreducible, note that if $v$ is the first vertex in the postorder of $T$, then $v$ cannot be in the right subtree of a vertex with two children, so $v$ and $\boxminus$ are both in $\gamma^{-1}(\boxminus)$.  

    The functions $\Psi$ and $\Psi^{-1}$ described above are inverses. Therefore, $\Psi$ is a bijection. It follows directly from the definition of $\Psi^{-1}$ that if $\Psi(\pi, (T_U)_{U \in \pi})=T$, then the multiset of branches $T_U$ for $U\in\pi$ is $\IF(T)$. This proves \eqref{eq:bpt-branch}, so \eqref{item2} is equivalent to \eqref{item3}.

    Finally, we will show that \eqref{item1} is equivalent to \eqref{item3}. Once again, fix $n\geq 1$ and $i_1,\ldots,i_n\in I$. 
    By \cite[Corollary~1.6]{Arizmendi2015}, we have
    \begin{equation}
        -K_n(X_{i_1}, \ldots, X_{i_n}) = \sum_{\substack{\sigma \in \Sn \\ \sigma(1) = n}} \prod_{U \in \runs(\sigma)} (- B_U(X_{i_1}, \ldots, X_{i_n})).
    \end{equation}
    Therefore, it suffices to show that 
    \begin{equation}\label{eq:dbpt-branch}
        \sum_{T \in \DBPT(i_1, \ldots, i_n)} \ttt(T) = \sum_{\substack{\sigma \in \Sn \\ \sigma(1) = n}} \prod_{U \in \runs(\sigma)} \sum_{T \in \Branch(i_U)} \ttt(T).
    \end{equation}
    Let $\Dn$ be the set of $\sigma \in \Sn$ with $\sigma(1) = n$ such that all the blocks of $\runs(\sigma)$ have cardinality at least $2$. Observe that the outer summand on the right-hand side of \eqref{eq:dbpt-branch} vanishes for $\sigma \not \in \Dn$. Let $Q(i_1, \ldots, i_n)$ be the set of pairs $(\sigma, (T_U)_{U \in \runs(\sigma)})$ with $\sigma \in \Dn$ and $T_U \in \Branch(i_U)$ for all $U \in \runs(\sigma)$. The right-hand side of \eqref{eq:dbpt-branch} can be expanded to
    \[\sum_{(\sigma, (T_U)_{U \in \runs(\sigma)}) \in Q(i_1, \ldots, i_n)}\, \prod_{U \in \runs(\sigma)} \ttt(T_U).\]
    Therefore, to prove \eqref{eq:dbpt-branch}, it suffices to exhibit a bijection $\Phi \colon Q(i_1, \ldots, i_n) \to \DBPT(i_1, \ldots, i_n)$ such that if $\Phi(\sigma, (T_U)_{U \in \runs(\sigma)}) = T$, then the multiset of branches $T_U$ for $U \in \runs(\sigma)$ is $\IF(T)$.

    Let $(\sigma, (T_U)_{U \in \runs(\sigma)}) \in Q(i_1, \ldots, i_n)$. For each $U\in\runs(\sigma)$, we may assume that the vertex set of the branch $T_U$ is $U\setminus\max U$ and that each non-root vertex of $T_U$ is smaller than its parent. The decreasing binary plane tree $\Phi(\sigma, (T_U)_{U \in \runs(\sigma)})$ can be constructed as follows. First, let $\sigma^\star \in \mathfrak{S}_{n - 1}$ be the permutation formed by deleting the entry $n$ (i.e., the first entry) from $\sigma$. Recall the bijection $\alpha \colon \DBPT_{n - 1} \to \mathfrak{S}_{n-1}$ from \cref{subsec:bpt}. Let $\alpha^{-1}(\sigma^*)=(T,\L)$. We may form a decreasing binary plane tree ${\tilde{\Phi}(\sigma, (T_U)_{U \in \runs(\sigma)}) \in \DBPT(i_1, \ldots, i_n)}$ from $(T,\L)$ by defining a coloring $\chi$ of $T$ such that $\chi(\boxminus) = i_n$ and such that $\chi(v) = i_{\L(v)}$ for all $v\in T$. Because $\sigma$ has no descending runs of length $1$, the tree $\tilde{\Phi}(\sigma, (T_U)_{U \in \runs(\sigma)})$ is a \dfn{reverse Motzkin tree}; that is, every vertex with a left child also has a right child. We will form $\Phi(\sigma, (T_U)_{U \in \runs(\sigma)})$ by modifying $\tilde{\Phi}(\sigma, (T_U)_{U \in \runs(\sigma)})$ as follows. Note that the non-leaf vertices of the branches $(T_U)_{U \in \runs(\sigma)}$ correspond to the vertices of $\tilde{\Phi}(\sigma, (T_U)_{U \in \runs(\sigma)})$ with exactly one child. For each vertex of one of the branches $(T_U)_{U \in \runs(\sigma)}$ that has a left child, swing $\tilde{\Phi}(\sigma, (T_U)_{U \in \runs(\sigma)})$ at the corresponding vertex. Let $\Phi(\sigma, (T_U)_{U \in \runs(\sigma)})$ be the result of performing all these swings.

    Inversely, let $T \in \DBPT(i_1, \ldots, i_n)$. We will now describe $\Phi^{-1}(T) \in Q(i_1, \ldots, i_n)$. First, let $\tilde{T}$ be the tree formed by swinging $T$ at each vertex with a left child but no right child. (Clearly, $\tilde{T}$ is a reverse Motzkin tree.) Let $\sigma$ be the result of appending $n$ to the beginning of $\alpha(\tilde{T})$. It is easy to verify that the vertices of $T$ with exactly one child are precisely those vertices with labels belonging to $U \setminus \{\min U, \max U\}$ for some $U \in \runs(\sigma)$. For each $U \in \runs(\sigma)$, form the branch $T_U \in \Branch(i_U)$ on the vertex set $U \setminus \{\max U\}$ by placing the elements of $U \setminus \{\max U\}$ in decreasing order from top to bottom. For each $k \in U \setminus \{\min U, \max U\}$, the vertex $k \in T_U$ has a left (respectively, right) child if and only if the vertex labeled $k$ in $T$ has a left (respectively, right) child. Let $\Phi^{-1}(T) = (\sigma, (T_U)_{U \in \runs(\sigma)})$.

    The functions $\Phi$ and $\Phi^{-1}$ described above are inverses. Therefore, $\Phi$ is a bijection. It follows directly from the definition of $\Phi^{-1}$ that if $\Phi(\sigma, (T_U)_{U \in \runs(\sigma)})=T$, then the multiset of branches $T_U$ for $U\in\pi$ is $\IF(T)$. This proves \eqref{eq:dbpt-branch}, so \eqref{item1} is equivalent to \eqref{item3}. This completes the proof.
\end{proof} 

\begin{example}\label{exam:Psi}
Let us illustrate the map $\Psi$ from the proof of \cref{thm:big_theorem}. Let $n=14$. Let $I$ be a $3$-element set whose elements are represented by the colors {\color{Cyan}cyan}, {\color{Yellow}yellow}, and {\color{Pink}magenta}. Let 
\[\pi=\{\{1,11,14\},\{2,3,8,9,10\},\{4,5,6,7\},\{12,13\}\}\in\NC_{\text{irr}}^{\geq 2}(14)\] be the partition shown on the top left of \cref{fig:Psi}, and let $i_j$ be the color of the circle containing the number $j$ in that drawing of $\pi$. Let $(T_U)_{U\in\pi}$ be as shown on the bottom left of \cref{fig:Psi}. Then $\Psi(\pi,(T_U)_{U\in\pi})$ is shown on the right of \cref{fig:Psi}.  
\end{example}

\begin{figure}
  \begin{center}{\includegraphics[height=9.794cm]{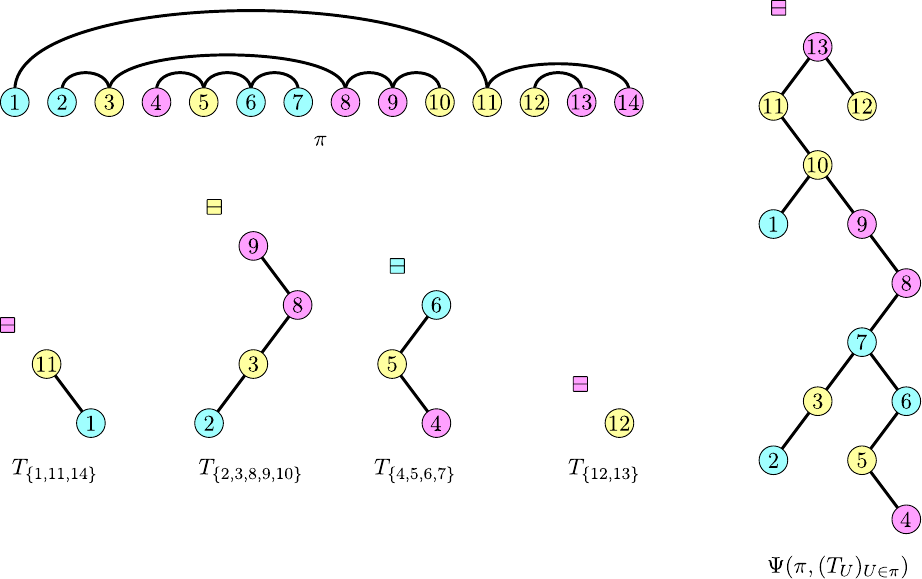}}
  \end{center}
\caption{On the left is a pair $(\pi,(T_U)_{U\in\pi})$, and on the right is its image under the map $\Psi$. Each number $j\in[14]$ is drawn in $\pi$ in a circle of color $i_j$.}\label{fig:Psi}
\end{figure}

\begin{example}\label{exam:Phi}
Let us illustrate the map $\Phi$ from the proof of \cref{thm:big_theorem}. Let $n=14$. Let $I$ and $(i_1,\ldots,i_{14})$ be as in \cref{exam:Psi}. Let \[\sigma=14\,\,\,6\,\,\,5\,\,\,9\,\,\,1\,\,\,13\,\,\,12\,\,\,10\,\,\,4\,\,\,2\,\,\,11\,\,\,8\,\,\,7\,\,\,3\] be the permutation shown on the top of \cref{fig:Phi}. Note that \[\runs(\sigma)=\{\{5,6,14\},\{1,9\},\{2,4,10,12,13\},\{3,78,11\}\}.\] Let $(T_U)_{U\in\runs(\sigma)}$ be as shown just below $\sigma$ in \cref{fig:Phi}. The reverse Motzkin tree $\tilde\Phi(\pi,(T_U)_{U\in\runs(\sigma)})$ is shown on the bottom left of \cref{fig:Phi}. Finally, the decreasing binary plane tree $\Phi(\pi,(T_U)_{U\in\runs(\sigma)})$ is shown on the bottom right of \cref{fig:Phi}. 
\end{example}

\begin{figure}
  \begin{center}{\includegraphics[height=14.717cm]{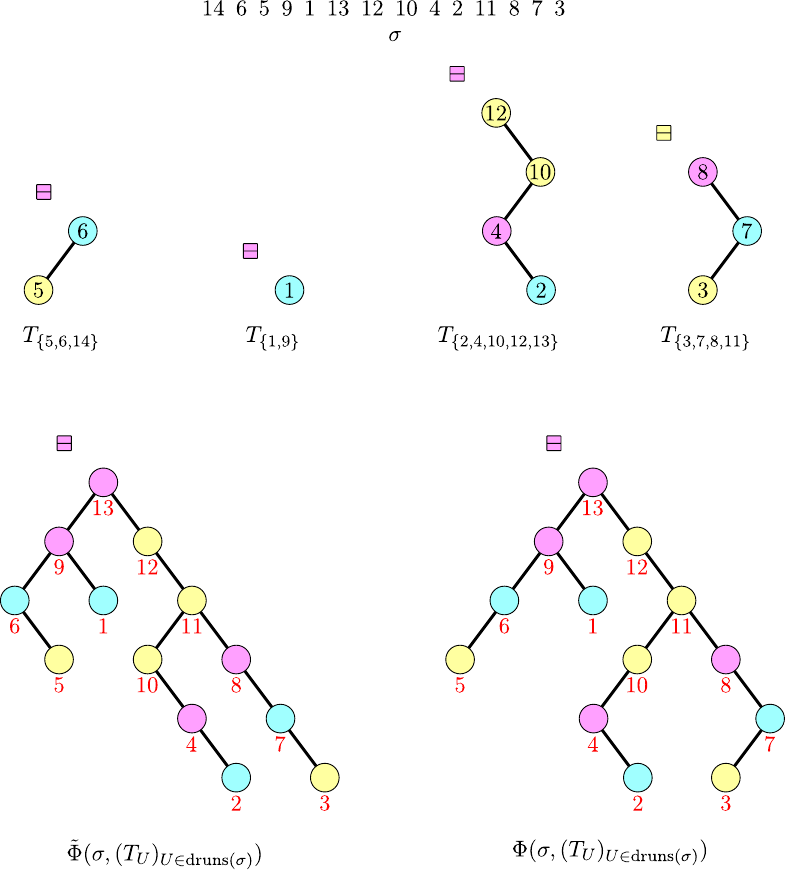}}
  \end{center}
\caption{On the top is a pair $(\sigma,(T_U)_{U\in\runs(\sigma)})$, and on the bottom are the trees $\tilde\Phi(\sigma,(T_U)_{U\in\runs(\sigma)})$ and $\Phi(\sigma,(T_U)_{U\in\runs(\sigma)})$. }\label{fig:Phi}
\end{figure}

\section{The Troupe Transform}\label{sec:troupe_transform} 
Let $\mathbf{T} \subseteq \BPT(I)$ be a troupe. For $n \geq 1$, define $\omega_n = |\mathbf{T} \cap \Branch_n|$ and $\widecheck{\omega}_n = |\mathbf{T} \cap \BPT_n|$.
The \dfn{troupe transform}, defined by Defant in 2022, is a function that transforms the sequence $(\omega_n)_{n \geq 1}$ to the sequence $(\widecheck{\omega}_n)_{n \geq 1}$ \cite[Section~2.5]{Defant2022}. In this section, we will show how to compute the troupe transform. 

By \cref{prop:if}, every sequence of nonnegative integers can arise as the sequence $(\omega_n)_{n \geq 1}$ for some choice of the troupe $\mathbf{T}$. Hence, to compute the troupe transform, it suffices to compute the sequence $(\widecheck{\omega}_n)_{n \geq 1}$ from $(\omega_n)_{n \geq 1}$. In fact, we will do so in the generality of weighted troupes.

\begin{theorem}\label{thm:troupe-transform}
    Let $\ttt \colon \BPT \to \mathbb K$ be a weighted troupe. For $n \geq 1$, let \[\omega_n = \sum_{T \in \Branch_n} \ttt(T)\quad\text{and}\quad\widecheck{\omega}_n = \sum_{T \in \BPT_n} \ttt(T).\]
    Then, the generating functions 
    \[\mathscr{B}(t) = \sum_{n \geq 1} \omega_n t^n\quad\text{and}\quad \mathscr T(t) = \sum_{n \geq 1} \widecheck{\omega}_n t^n\] satisfy \[\mathscr T(t) = \mathscr B\left(\frac{t}{1 - t\mathscr T(t)}\right).\]
\end{theorem}

\begin{proof}
    Let $\A = \mathbb K[X]$. We may then choose the functional $\varphi \colon \A \to \mathbb K$ in such a way that $\varphi(X) = 0$ and, for all $n \geq 2$, we have
    \begin{equation}\label{eq:bomega}-B_n(X, \ldots, X) = \omega_{n - 1}\end{equation} in the noncommutative probability space $(\A, \varphi)$.
    By \cref{thm:big_theorem}, we have
    \begin{equation}\label{eq:romega}-R_n(X, \ldots, X) = \widecheck{\omega}_{n - 1}\end{equation} for all $n \geq 2$.

     By \cite[Equation~2.16]{Arizmendi2015}, we have \begin{equation}\label{eq:br}1 - B\left(\frac{t}{1 + R(t)}\right) = \frac{1}{1 + R(t)},\end{equation} where \[B(t) = \sum_{n \geq 1} B_n(X, \ldots, X) t^n\] is the ordinary generating function for the Boolean cumulants of $X$ and \[R(t) = \sum_{n \geq 1} R_n(X, \ldots, X) t^n\] is the ordinary generating function for the free cumulants of $X$. It follows from \eqref{eq:bomega} and \eqref{eq:romega} that $B(t) = - t \mathscr{B}(t)$ and $R(x) = -t \mathscr{T}(t)$; substituting into \eqref{eq:br} yields the desired result.
\end{proof}

\begin{corollary}\label{cor:algebraic}
    Let $\ttt \colon \BPT \to \mathbb K$ be a weighted troupe. For $n \geq 1$, let \[\omega_n = \sum_{T \in \Branch_n} \ttt(T)\quad\text{and}\quad\widecheck{\omega}_n = \sum_{T \in \BPT_n} \ttt(T).\]
    Let \[\mathscr B(t) = \sum_{n \geq 1} \omega_n t^n\quad\text{and}\quad \mathscr T(t) = \sum_{n \geq 1} \widecheck{\omega}_n t^n.\] Then $\mathscr B(x)$ is algebraic over $\mathbb K[t]$ if and only if $\mathscr T(t)$ is. 
\end{corollary}
\begin{proof}
    Suppose that $\mathscr B(t)$ is algebraic (over $\mathbb{K}[t]$). This means that there exists a nonzero bivariate polynomial $P$ with coefficients in $\mathbb K$ such that $P(t, \mathscr B(t)) = 0$. By \cref{thm:troupe-transform}, we have \[P\left(\frac{t}{1 - t \mathscr{T}(t)}, \mathscr{T}(t)\right) = P\left(\frac{t}{1 - t \mathscr{T}(t)}, \mathscr{B}\left(\frac{t}{1 - t \mathscr{T}(t)}\right)\right) = 0.\] By clearing the denominators in this equation, we find a nonzero bivariate polynomial $Q$ with coefficients in $\mathbb K$ such that $Q(t, \mathscr{T}(t)) = 0$. Therefore, $\mathscr{T}(t)$ is algebraic.

    Conversely, suppose that $\mathscr{T}(t)$ is algebraic. Then, so is the power series \[\mathscr{W}(t) = \frac{t}{1 - t \mathscr{T}(t)}.\] By \cref{thm:troupe-transform}, we have $\mathscr{T}(t) = \mathscr{B}(\mathscr W(t))$, so $\mathscr{B} = \mathscr{T} \circ \mathscr W^{\langle -1 \rangle}$, where $\mathscr W^{\langle -1 \rangle}$ denotes the compositional inverse of $\mathscr W$. It is well known that the compositional inverse of an algebraic power series is algebraic and that the composition of two algebraic power series is algebraic. Therefore, $\mathscr{B}$ is algebraic.
\end{proof}

\cref{thm:troupe-transform,cor:algebraic} answer \cite[Question~9.1]{Defant2022}. 

\section{Applications to Permutation Enumeration}\label{sec:peaks}  

Fix $i_1, \ldots, i_n\in I$, and recall that there is a bijection $\alpha\colon\DBPT(i_1, \ldots, i_n)\to\Sn$ that maps a decreasing binary plane tree $T$ to the permutation formed by reading the labels of the vertices of $T$ in inorder. Using this bijection, we can rewrite the sum from \cref{thm:big_theorem}\eqref{item1} as a sum over permutations rather than a sum over decreasing binary plane trees. This motivates us to understand the insertion factors of a decreasing binary plane tree $(T,\L)\in\DBPT(i_1,\ldots,i_n)$ in terms of the permutation $\alpha(T,\L)$. We consider each insertion factor as a decreasing binary plane tree, where the labeling is just the restriction of $\L$. 

The \dfn{plot} of a word $w$ over the alphabet $\mathbb Z$ is the diagram that depicts the points $(i,w(i))$ for all $i\in[n]$. It turns out that there is a simple way to read off the insertion factors of $\alpha^{-1}(w)$ from the plot of $w$. A \dfn{peak} of $w$ is an index $p\in\{2,\ldots,n-1\}$ such that $w(p-1)<w(p)>w(p+1)$. Let $p_1<\cdots<p_\ell$ be the peaks of $w$. For $1\leq j\leq \ell$, let $\sw_j(w)$ be the set of points in the plot of $w$ that lie weakly southeast of $(p_j,w(p_j))$ but do not lie weakly southeast of any of the points $(p_{j'},w(p_{j'}))$ for $j<j'\leq\ell$. Let $\sw_0(w)$ be the set of points in the plot of $w$ that do not belong to $\bigcup_{j=1}^\ell\sw_j(w)$. After normalizing, we can view each set $\sw_j(w)$ as the plot of a word $w_j$. Let $v_j$ be the vertex of $\alpha^{-1}(w)$ with label $w(p_j)$. Then $w_j$ is obtained by reading the labels of the insertion factor $T_{v_j}$ in inorder; this uniquely determines $T_{v_j}$. 

\begin{example}
\cref{fig:peak_diagram} shows the plot of the permutation 
\[w=15\,\,\,16\,\,\,10\,\,\,11\,\,\,6\,\,\,20\,\,\,18\,\,\,12\,\,\,1\,\,\,7\,\,\,13\,\,\,17\,\,\,8\,\,\,3\,\,\,2\,\,\,9\,\,\,5\,\,\,4\,\,\,14\,\,\,19.\] The $5$ peaks of this permutation are $p_1=2$, $p_2=4$, $p_3=6$, $p_4=12$, $p_5=16$. For $1\leq j\leq 5$, the region that is southeast of $(p_j,w(p_j))$ and not southeast of any of the points $(p_{j'},w(p_{j'}))$ for $j<j'\leq 5$ is shaded in some color. \cref{fig:big_insertion_factors} shows the tree $\alpha^{-1}(w)$ on the left and its insertion factors on the right.  
\end{example}

\begin{figure}
  \begin{center}{\includegraphics[width=\linewidth]{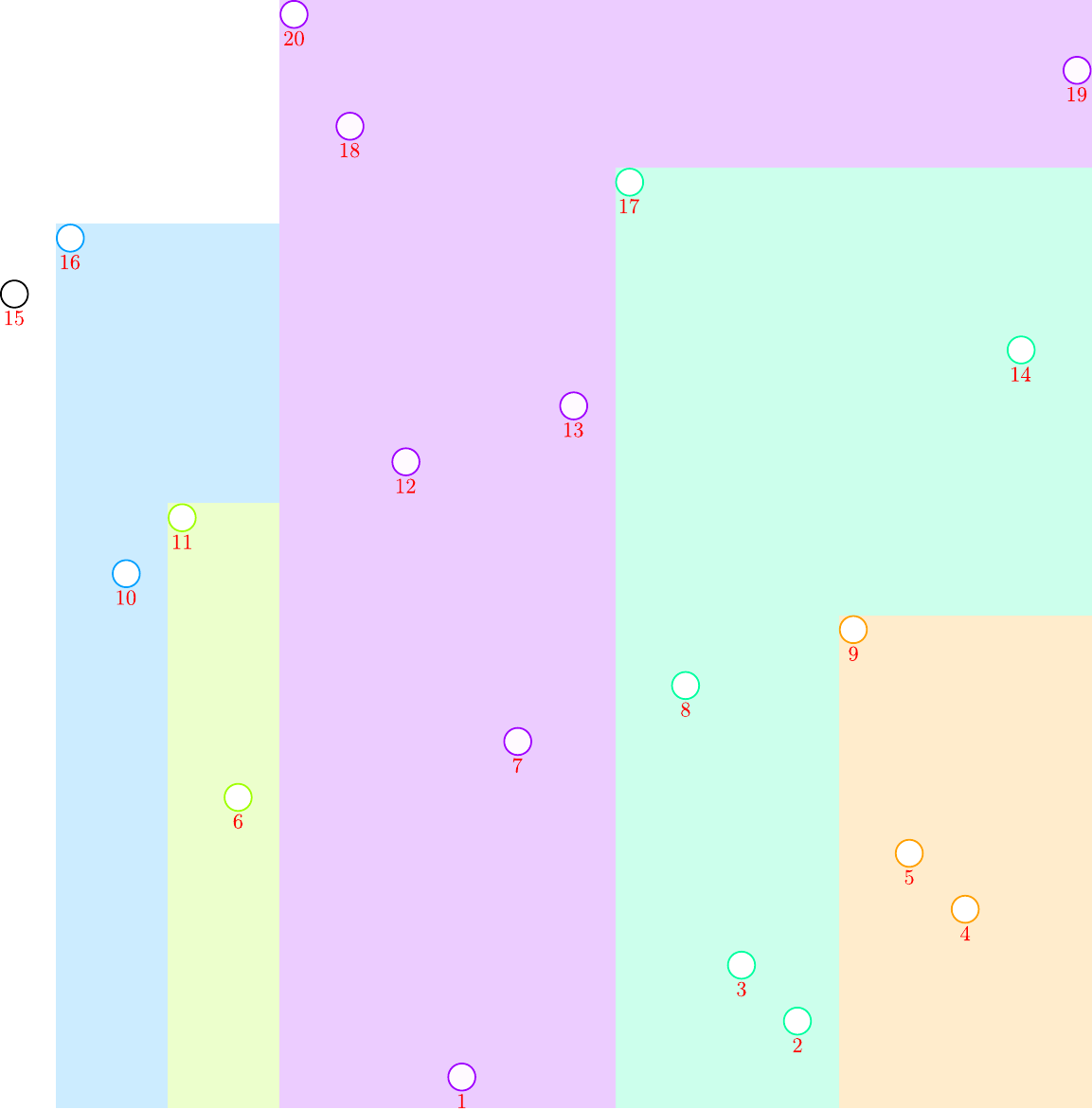}}
  \end{center}
\caption{The plot of a permutation $w$. Each set $\sw_j(w)$ for $j\geq 1$ consists of the points in the plot lying in one of the colored regions. The set $\sw_0(w)$ consists of the points in the plot that do not lie in any of the colored regions; in this example, $\sw_0(w)$ contains just the point $(1,15)$.}\label{fig:peak_diagram} 
\end{figure} 

\begin{figure}
  \begin{center}{\includegraphics[width=0.846\linewidth]{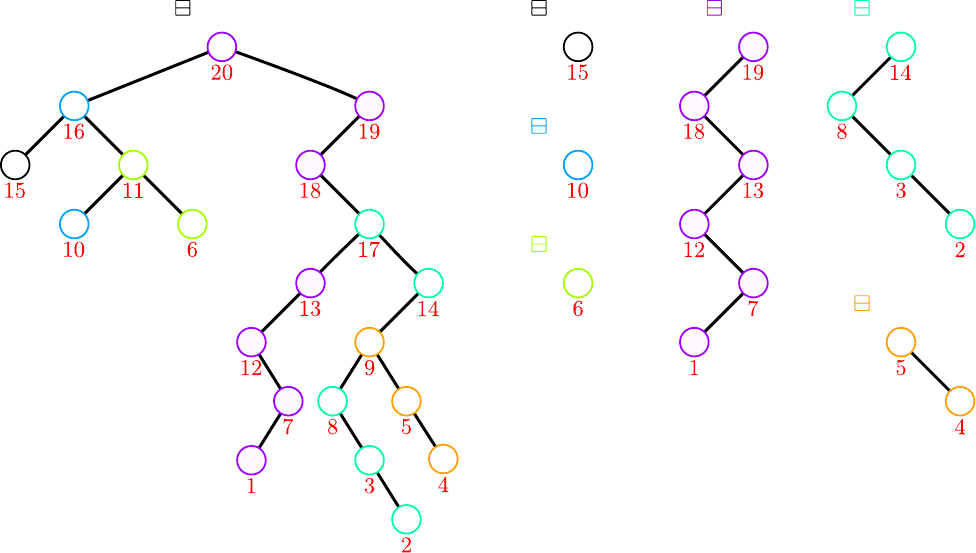}}
  \end{center}
\caption{A decreasing binary plane tree (left) and its insertion factors (right).}\label{fig:big_insertion_factors} 
\end{figure} 

\section{Cumulants from Distributions}\label{sec:distributions}
Let us use the setup of \cref{thm:big_theorem}. Suppose that $I = \{\star\}$ is a singleton, that $\mathbb K = \mathbb{R}$, and that $\A = \mathbb{R}[X]$, where $X = X_{\star}$. Further, for any rapidly decreasing (Schwartz) distribution $f$ on $\mathbb{R}$ with total mass $1$, define the linear functional $\varphi_f \colon \A \to \mathbb{R}$ by
\[\varphi_f(X^n) = \int_{-\infty}^\infty x^n f(x) \, dx\]
for all $n \geq 0$. Then $(\A, \varphi_f)$ is a commutative probability space. If $f$ is a probability distribution (i.e., it is nonnegative), then $(\A, \varphi_f)$ is genuinely a space of random variables.

When we refer to the $n$th classical cumulant of a distribution $f$, we mean the classical cumulant $K_n(X,\ldots,X)$ computed in the probability space $(\A,\varphi_f)$. We use the notation $\kappa_n(f)$ for this quantity. Similarly, we define the $n$th free and Boolean cumulants of the distribution $f$, denoted $r_n(f)$ and $b_n(f)$, respectively.

Classical cumulants satisfy the fundamental property that for all rapidly decreasing distributions $f, g$, we have \begin{equation}\label{eq:kappa-ast}\kappa_n(f \ast g) = \kappa_n(f) + \kappa_n(g),\end{equation} where $\ast$ denotes convolution.

We will consider some particular examples of weighted troupes $\ttt \colon \BPT \to \mathbb{R}$. For each such example, we attempt to exhibit a distribution $f$ such that the three conditions of \cref{thm:big_theorem} actually hold in the probability space $(\A, \varphi_f)$. In other words, we want the following three equivalent conditions to hold for all $n \geq 1$:
\begin{align}
    \label{eq:kappa}\kappa_n(f) &= - \sum_{T \in \DBPT_{n-1}} \ttt(T), \\
    \label{eq:r}r_n(f) &= - \sum_{T \in \BPT_{n-1}} \ttt(T), \\
    \label{eq:b}b_n(f) &= - \sum_{T \in \Branch_{n-1}} \ttt(T).
\end{align}
(We remind the reader that if $I = \{\star\}$ is a singleton, then the set $\DBPT(i_1, \ldots, i_n)$ appearing in the statement of \cref{thm:big_theorem} is $\DBPT_{n - 1}$, not $\DBPT_n$. Similarly, $\BPT(i_1, \ldots, i_n) = \BPT_{n - 1}$ and $\Branch(i_1, \ldots, i_n) = \Branch_{n - 1}$.) The distribution $f$ will not necessarily be a probability distribution.

\subsection{Binary Plane Trees}\label{subsec:table1}
Suppose that $\ttt$ is the indicator function of $\BPT \setminus \{\varnothing\}$. Then, for all $n \geq 1$, we have
\begin{align*}
\sum_{T \in \DBPT_{n}} \ttt(T) &= |\DBPT_n| = n!, \\
\sum_{T \in \BPT_{n}} \ttt(T) &= |\BPT_n| = C_{n}, \\ 
\sum_{T \in \Branch_{n}} \ttt(T) &= |\Branch_n| = 2^{n - 1}.
\end{align*}
Hence, the condition \eqref{eq:kappa} becomes
\[\kappa_n(f) = \begin{cases}-(n-1)! & \mbox{if $n > 1$} \\ 0 & \mbox{if $n = 1$}\end{cases}\] for all $n \geq 1$.
This is impossible if $f$ is a probability distribution. Indeed, if $f$ is a probability distribution, then $\kappa_2(f)$ is its variance, which is nonnegative and thus cannot be equal to $-1$. However, it is possible to construct a distribution $f$ satisfying \eqref{eq:kappa}, as follows.

Let
\[g(x)=\begin{cases}e^{-x+1} & \mbox{if $x \geq -1$} \\ 0 & \mbox{otherwise.}\end{cases}\] This is the standard exponential distribution shifted so that its mean is $0$. It is well known that the $n$th classical cumulant of the standard exponential distribution is $(n - 1)!$ \cite{Rota2000}, so 
\[\kappa_n(g) = \begin{cases}(n-1)! & \mbox{if $n > 1$} \\ 0 & \mbox{if $n = 1$.}\end{cases}\] Now, let \[f(x) = \delta(x-1) + \delta'(x-1),\] where $\delta$ denotes the Dirac delta distribution. It is not difficult to check that $f$ is the convolution inverse of $g$; i.e. $f \ast g = \delta$. By \eqref{eq:kappa-ast}, we have $\kappa_n(f) = - \kappa_n(g)$ for all $n \geq 1$, so $f$ indeed satisfies \eqref{eq:kappa}. By \cref{thm:big_theorem}, $f$ satisfies \eqref{eq:r} and \eqref{eq:b} as well.

%

Note that for every integer $k \geq 0$, the $k$-fold convolution $g^{\ast k}$ is the well known \emph{gamma distribution} with shape parameter $k$, shifted to have mean $0$. Thus, we may informally think of $f$ as the ``gamma distribution with shape parameter $-1$'', shifted to have mean $0$.

\subsection{Binary Plane Trees Weighted by Right Edges}\label{subsec:table2}
Fix a real number $q$. Consider the weighted troupe $\ttt\colon\BPT\to\mathbb R$ defined by 
$\ttt(T)=q^{\rr(T)+1}$ for all nonempty $T$, where $\rr(T)$ is the number of right edges in $T$. If $q=1$, then this coincides with the weighted troupe considered in \cref{subsec:table1}; hence, we will assume in what follows that $q\neq 1$. For all $n\geq 1$, we have 
\begin{align*}
\sum_{T \in \DBPT_{n}} \ttt(T) & = qA_n(q), \\
\sum_{T \in \BPT_{n}} \ttt(T) &= qN_{n}(q), \\ 
\sum_{T \in \Branch_{n}} \ttt(T) &= q(1+q)^{n-1},
\end{align*}
where $A_n(t)$ and $N_n(t)$ denote the $n$th \emph{Eulerian polynomial} and the $n$th \emph{Narayana polynomial}, respectively \cite[Example~6.3]{Defant2022}.

Let \[f(x) = \frac{\delta(x - q) - q \delta(x - 1)}{1 - q}.\] We can check that $f$ satisfies \eqref{eq:kappa} directly as follows.
The moment-generating function of $f$ is \[M_X(t) = \sum_{n \geq 0} \varphi_f(X^n) \frac{t^n}{n!} = \int_{-\infty}^\infty e^{tx} f(x)\, dx = \frac{e^{qt} - q e^t}{1 - q}.\] Therefore, by \eqref{eq:classical-egf}, we have \[\sum_{n \geq 0} \kappa_n(f) \frac{t^n}{n!} = \log\left(\frac{e^{qt} - q e^t}{1 - q}\right).\] From this, it is straightforward to check that 
\[\kappa_n(f) = \begin{cases} -q A_{n-1}(q) & \mbox{if $n > 1$} \\ 0 & \mbox{if $n = 1$,}\end{cases}\] so $f$ satisfies \eqref{eq:kappa}, as desired.  By \cref{thm:big_theorem}, $f$ satisfies \eqref{eq:r} and \eqref{eq:b} as well.

Suppose that $|q| < 1$. Motivated by \cref{subsec:table1}, we find that the convolution inverse of $f$ is
\[g(x) = (1 - q) \sum_{m \geq 0} q^m \delta(x + q - (1 - q)m).\] The distribution $g$ satisfies 
\begin{align*}
    \kappa_n(g) = \sum_{T \in \DBPT_{n - 1}} \tau(T) = \begin{cases} -q A_{n-1}(q) & \mbox{if $n > 1$} \\ 0 & \mbox{if $n = 1$}.\end{cases}
\end{align*}
If $0 \leq q < 1$, then $g$ is a shifted and rescaled geometric distribution with success probability $1 - q$.

\subsection{Full Binary Plane Trees} 
In each of the previous two examples (\cref{subsec:table1,subsec:table2}), we found two distributions $f$ and $g$ which were convolution inverses of each other. The distribution $f$ was chosen to satisfy \eqref{eq:kappa}; that is, 
\[\kappa_n(f) = - \sum_{T \in \DBPT_{n - 1}} \ttt(T)\] for all $n \geq 1$. (By \cref{thm:big_theorem}, $f$ satisfied \eqref{eq:r} and \eqref{eq:b} as well.) On the other hand, the distribution $g$ satisfied the equation \begin{equation}\label{eq:kappa-g}\kappa_n(g) = \sum_{T \in \DBPT_{n - 1}} \ttt(T)\end{equation} for all $n \geq 1$, which differs only in that it has no negative sign. 

Recall that a nonempty binary plane tree is \emph{full} if no vertex has exactly one child. Let $\FBPT$ be the troupe of full binary plane trees. Let $\ttt$ denote the indicator function of $\FBPT$. For all $n\geq 1$, we have 
\begin{align*}
\sum_{T \in \DBPT_{n}} \ttt(T) &= \begin{cases} \Tan_{n} & \mbox{if $n$ is odd} \\ 0 & \mbox{if $n$ is even},\end{cases} \\
\sum_{T \in \BPT_{n}} \ttt(T) &= \begin{cases}C_{(n - 1)/2} & \mbox{if $n$ is odd} \\ 0 & \mbox{if $n$ is even},\end{cases} \\ 
\sum_{T \in \Branch_{n}} \ttt(T) &= \begin{cases}1 & \mbox{if $n = 1$} \\ 0 & \mbox{otherwise},\end{cases}
\end{align*}
where $\Tan_{n}$ is the number of alternating permutations of size $n$ \cite[Example~4.8]{Defant2022}; the numbers $\Tan_1,\Tan_3,\Tan_5,\ldots$ are also known as \dfn{tangent numbers} since they are the coefficients of the Taylor series for the tangent function.

Let $\ttt$ be the indicator function of $\FBPT$. We will still be able to exhibit a distribution $g$ satisfying \eqref{eq:kappa-g}. However, in contrast to \cref{subsec:table1,subsec:table2}, the distribution $g$ does not have a convolution inverse, so there is no distribution $f$ satisfying \eqref{eq:kappa} (and by \cref{thm:big_theorem}, there is no distribution $f$ satisfying \eqref{eq:r} or \eqref{eq:b} either).

Let
\[g(x) = \frac{1}{\exp\left(\frac{\pi x}{2}\right) + \exp\left(-\frac{\pi x}{2}\right)};\] then $g$ is the hyperbolic secant distribution. It is well known \cite[Section~1.3]{Fischer2014} that the moment-generating function of $g$ is \[M_X(t) = \sum_{n \geq 0} \varphi_g(X^n) \frac{t^n}{n!} = \int_{-\infty}^{\infty} e^{tx} g(x) \, dx = \sec t,\] so by \eqref{eq:classical-egf}, we have \[\sum_{n \geq 0} \kappa_n(g) \frac{t^n}{n!} = \log(\sec t).\] Differentiating with respect to $t$, we have \[\sum_{n \geq 1} \kappa_{n}(g) \frac{t^{n - 1}}{(n - 1)!} = \frac{d}{dt}\log(\sec t) = \tan t,\] so $\kappa_n(g)$ is the coefficient of $t^{n - 1} / (n - 1)!$ in the Taylor series for $\tan t$. That is, \[\kappa_n(g) = \begin{cases} \Tan_{n} & \mbox{if $n$ is odd} \\ 0 & \mbox{if $n$ is even},\end{cases}\] so $g$ satisfies \eqref{eq:kappa-g} as desired.

It is unclear whether the hyperbolic secant distribution $g$ has any combinatorial meaning related to the troupe $\FBPT$.


\section*{Acknowledgments}
Colin Defant was supported by the National Science Foundation under Award No.\ 2201907 and by a Benjamin Peirce Fellowship at Harvard University.

\bibliographystyle{alphaurl}
\bibliography{main}

\end{document}